\documentclass[11pt,a4paper]{article}
\usepackage{amsmath}
\usepackage{amsfonts}
\usepackage{amssymb}
\usepackage{graphicx}
\usepackage{amsthm}
\usepackage{ragged2e}
\usepackage{hyperref}
\usepackage[left=2cm,right=2cm,top=2cm,bottom=2cm]{geometry}
\author{Dimitris Lygkonis, Vassilis Nestoridis}
\title{Localized versions of function spaces and generic results}
\date{}

\newcommand{\dist}{\text{dist}}
\newcommand{\tct}{\overline{\phantom{.}{\overline{\OO}}^{\mathrm{c}}\phantom{.}}}
\newcommand{\bb}{b_1\subset\overline b_1\subset b_2\cap\OO, \ \ b_2\cap\OO^c\neq\emptyset}
\newcommand{\zn}{(z_n)_{n \in \N} }
\newcommand{\xn}{(x_n)_{n \in \N} }
\newcommand{\wn}{(w_n)_{n \in \N} }
\newcommand{\tn}{(\tau_n)_{n \in \N} }
\newcommand{\dn}{(\de_n)_{n \in \N} }

\newcommand{\ch}{\overline}
\newcommand{\dif}{\smallsetminus}
\newcommand{\tomi}{\displaystyle\bigcap}
\newcommand{\enosi}{\displaystyle\bigcup}
\newcommand{\ext}{\widehat{\Co}}
\newcommand{\ct}{\mathrm{c}}
\newcommand{\bd}{\partial}
\newcommand{\OO} {{\Omega}}
\newcommand{\Co}{\mathbb{C}}

\newcommand{\N}{\mathbb{N}}

\newcommand{\R}{\mathbb{R}}
\newcommand{\ra}{\rightarrow}
\newcommand{\lra}{\longrightarrow}

\newcommand{\norm}[1]{\left\lVert#1\right\rVert}

\newcommand{\e}{\epsilon}
\newcommand{\de}{\delta}

\theoremstyle{definition}
\newtheorem{theorem}{Theorem}[section]
\newtheorem{lemma}[theorem]{Lemma}
\newtheorem{proposition}[theorem]{Proposition}

\newtheorem{corollary}[theorem]{Corollary}
\newtheorem{definition}[theorem]{Definition}

\newtheorem{example}[theorem]{Example}

\newtheorem{remark}[theorem]{Remark}

\numberwithin{equation}{section}

\begin{document}
\maketitle
\begin{flushleft}

\begin{abstract}
\justify
We consider generalizations of classical function spaces by requiring that a holomorphic in $\Omega$ function  satisfies some property when we approach from $\Omega$, not the whole boundary $\partial{\Omega}$, but only a part of it. These spaces endowed with their natural topology are Fr\'{e}chet spaces. We prove some generic non-extendability results in such spaces and generic nowhere differentiability on the corresponding part of $\partial{\Omega}$. \\
\end{abstract}

{\footnotesize AMS classification number: 30H05, 30H20, 30H50 }

\section{Introduction}
\justify
Let $\OO$ be a domain in $\mathbb{C}$ or $\mathbb{C}^d$. Let $X(\OO)$ be a set of holomorphic functions in $\Omega$ which is a Fr\'{e}chet space. We also assume that the convergence $f_{n} \longrightarrow f$ in the topology of $X(\OO)$ implies the pointwise convergence $f_{n}(z) \longrightarrow f(z)$ for all $z \in \OO$. In order that there exists a non-extendable function, $f$ in $X(\OO)$, it suffices that the following holds: For every pair of open balls $(b_1,b_2)$, satisfying $ b_1 \subset \overline{b_1} \subset b_2 \cap \OO$ and $b_2 \cap  {\OO}^{\mathrm{c}} \neq \emptyset$, there exists a function $f=f_{b_1,b_2}$  in $X(\OO)$, such that $f_{|b_{1}}$ does not admit any bounded holomorphic extensions on $b_2$. Furthermore, if the above hold, the set $\{f \in X(\OO) : f \: \text{is non-extendable}\}$ is dense and $G_\delta$ in $X(\OO)$ \cite{DoH}.
\par Examples of functions spaces $X(\OO)$ satisfying the above assumptions include most of the classical functions spaces, as $H(\OO), \: 
A(\OO),\: A^{p}(\OO),\: H^{p}(\OO)$, Bergman spaces etc. Most of these spaces are defined as the set of holomorphic in $\OO$ functions, satisfying some additional property when we approach the whole boundary $\partial{\OO}$ from $\OO$. We can generalize these spaces by requiring a property to hold when we approach only a part $J$ of $\partial{\OO}$ and consider combinations of such spaces. Then, these spaces endowed with their natural topology are also Fr\'{e}chet spaces, which satisfy the above assumptions. We can investigate non-extendability of functions belonging to these spaces. The natural assumption is that the part $J$, is a relatively open subset of the boundary $\partial{\OO}$.
\par A first example is the space $X(\OO,V)= H(\OO)\cap H^{\infty}(V)$ containing all holomorphic in $\Omega$ functions $f$ bounded on $V$, where $V$ is an open subset of $\OO$. Then, the natural topology is induced by the seminorms $  \norm{f|_{K_m}}_{\infty}$ and $\norm{f|_{V}}_{\infty}$, where, $\{K_m\}_{m=1}^{\infty}$ is an exhaustive sequence of compact subsets of $\OO$ \cite{Ru}. We prove that if $\overline{V} \cap \bd{\OO}$ is contained in $  \overline{\phantom{.}{\overline{\OO}}^{\mathrm{c}}\phantom{.}}$, then there exist non-extendable functions in $X(\OO,V)$ and their set is dense and $G_\delta$. Here, $\OO \subset \mathbb{C}$, but we also discuss some extensions for $\OO \subset \mathbb{C}^d$. 
\par Next, we generalize the Bergman spaces considering holomorphic in $\OO$ functions f, integrable on $V$ and we prove similar results. Variations of the previous spaces are obtained by requiring $f^{(l)}$ for $l$ in some set $F \subset \{0,1,2,...\}$ satisfy the previous requirements. We can also consider several subsets $V_i, i \in I$, where $I$ is a finite or infinite denumerable set, and consider the space of holomorphic in $\OO$ functions satisfying different properties in each $V_i$. 
\par If $\OO$ is a Jordan domain and $J \subset \bd{\OO}$ is relatively open, we consider the space $A^{0}(\OO,J)$ to contain all holomorphic in $\OO$ functions extending continuously on $\OO \cup J$, endowed with its natural topology, see also \cite{GeorgMastrNest}. We show that the generic function in $A^{0}(\OO,J)$ is nowhere differentiable on $J$. Here, the differentiability is meant with respect to the parametrization induced by any Riemann map from the open unit disc onto $\OO$ \cite{MastrPan}, or with respect to the position \cite{KavvMakr}. We notice that in this case, polynomials are dense in $A^{0}(\OO,J)$. Furthermore, we generalize the previous results to domains $\OO$ bounded by a finite number of disjoint Jordan curves. We also consider the spaces $A^{p}(\OO,J)$ containing all functions $f \in A^{0}(\OO,J)$, such that all the derivatives $ f^{(l)}, \: 0 \leq l\leq p$ belong to $A^{0}(\OO,J)$, endowed with its natural topology. We show that if $\OO$ is convex, then for the generic function $f \in A^{p}(\OO,J)$, the derivative $f^{(p)}$ is nowhere differentiable on $J$.

\section{Preliminaries}

\begin{definition} \label{2.1.}
Let  $\OO \subset \mathbb{C}^{d}$ be open and connected and $f: \OO \rightarrow \mathbb{C}$  be a holomorphic function. Then, $f$ is called extendable if there exist an open and connected set $U \subset \mathbb{C}^d$ with $U \cap \OO \neq \emptyset$ and $U \cap \OO^c \neq \emptyset$, a holomorphic function $F : U \rightarrow \mathbb{C}$ and a component $V$ of $U \cap \OO$ such that $F_{|V} = f_{|V}$. Otherwise, $f$ is called non-extendable \cite{Range}.

\end{definition}

\begin{definition} \label{2.2.}
Let $\OO \subset\Co^d$ be open and connected and $f:\OO\rightarrow \Co$ be a holomorphic function. Then, $f$ is called extendable in the sense of Riemann domains, if there exist two open Euclidean balls $b_1,b_2\subset\Co^d$, with
\[
b_1\subset\overline b_1\subset b_2\cap\OO, \ \ b_2\cap\OO\neq\emptyset, \ \ b_2\cap\OO^c\neq\emptyset,
\]
and a bounded holomorphic function $F:b_2\rightarrow\Co$ such that $F|_{b_1}=f|_{b_1}$. Otherwise the function $f$ is called non-extendable in the sense of Riemann domains \cite{DoH}.

\end{definition}

\begin{proposition} \label{2.3.}
Definitions 2.1. and 2.2. are equivalent.
\end{proposition}

\begin{proof}
The proof of Proposition 2.3. is contained in \cite{DoH}.
\end{proof}

\noindent Let $\OO\subset \mathbb{C}^d$ be open and connected and let $X=X(\OO)$ be a subset of $H(\OO)$.

\begin{definition} \label{2.4.}
The open connected set $\OO\subset\Co^d$ is called a $X$-domain of holomorphy if there exists $f\in X$ which is non-extendable \cite{DoH}.
\end{definition}

\begin{definition} \label{2.5.}
The open connected set $\OO\subset\Co^d$ is called weak $X$-domain of holomorphy if for every pair of open Euclidean balls $b_1,b_2$ with $b_2\cap\OO\neq\emptyset$, $b_2\cap\OO^c\neq\emptyset$, $b_1\subset\overline{b}_1\subset b_2\cap\OO$ there exists a function $f_{b_1,b_2}\in X$ such that the restriction of $f_{b_1,b_2}$ on $b_1$ does not have any bounded holomorphic extension on $b_2$ \cite{DoH}.
\end{definition}

\begin{theorem} \label{2.6.}
We suppose that $X=X(\OO)\subset H(\OO)$ is a topological vector space endowed with the usual operations $+,\cdot$ and that its topology is induced by a complete metric. We also suppose that the convergence $f_n\lra f$ in $X$ implies the pointwise convergence $f_n(z)\lra f(z)$ for all $z\in\OO$. Then, definitions 2.4. and 2.5. are equivalent. If the above assumptions hold and $\OO$ satisfies definitions 2.4. and 2.5., then the set $\{f\in X: f$ is non extendable$\}$ is a dense and $G_\delta$ subset of $X$.
\end{theorem}

\begin{proof}
The proof of Theorem 2.6. is contained in \cite{DoH}.
\end{proof}

It follows that in order to prove that $\OO$ is a $X(\OO)$-domain of holomorphy, it suffices to prove the following: For every pair of euclidean balls $(b_1,b_2)$, such that $b_2\cap\OO\neq\emptyset$, $b_2\cap\OO^c\neq\emptyset$, $b_1\subset\overline{b}_1\subset b_2\cap\OO$ there exists a function $f_{b_1,b_2}\in X$ such that the restriction of $f_{b_1,b_2}$ on $b_1$ does not have any bounded holomorphic extension on $b_2$ \cite{DoH}.

\begin{lemma} \label{2.7.}
Let $\gamma$ be a Jordan curve, $J \subset \gamma$ a rectifiable open arc and $J' \subset J$ a compact arc. Then, $J'$ can be extended to a rectifiable Jordan curve $\gamma'$, such that the interior of $\gamma'$ is a subset of the interior of $\gamma$. 
\end{lemma}

\begin{proof}
The proof of lemma 2.7. is contained in \cite{LioNest}.
\end{proof}

\begin{lemma} \label{2.8.}
Let $\OO$ be a domain and $J \subset \partial{\OO}$ a relatively open subset of its boundary. Suppose that $\ext \dif \overline{\OO}$ is connected and that $\bd{\OO} \dif J$ is contained in $\overline{\phantom{.}{\overline{\OO}}^{\mathrm{c}}\phantom{.}}$. Then, for every $m \in \N$, the set $\Delta_m=\{z \in \OO \cup J : \dist(z,\bd{\OO} \dif J) \geq \frac{1}{m}\}$ has connected complement in $\ext$, where $\ext=\Co \cup \{\infty\}$.
\end{lemma}

\begin{proof}
Let $m \in \N$. The set $\Delta_m$ can be written as follows: $\Delta_m= \tomi_{w \in \bd{\OO} \dif J} D(w,\frac{1}{m})^{\ct} \cap (\OO \cup J)$. Therefore, we have that $\ext \dif \Delta_m =  \enosi_{w \in \bd{\OO} \dif J} D(w,\frac{1}{m}) \cup (\ext \dif \overline{\OO}) \cup (\bd{\OO} \dif J)$. The difference $\bd{\OO} \dif J$ is contained in the union of $D(w,\frac{1}{m}), w \in \bd{\OO} \dif J$, hence $\ext \dif \Delta_m =  \enosi_{w \in \bd{\OO} \dif J} D(w,\frac{1}{m}) \cup (\ext \dif \overline{\OO})$. Since every open disc is connected, $\ext \dif \overline{\OO}$ is connected and intersects every open disc $D(w,\frac{1}{m}), w \in \bd{\OO} \dif J$, we conclude that the set $\ext \dif \Delta_m$ is also connected.
\end{proof}

\begin{definition} \label{2.9.}
Let $L \subset \bd{D}$ be a relatively open subset of the unit circle. We say that a continuous function $f \in C(L)$ belongs to $Z_L$, if for every $\theta \in L$ we have that $\displaystyle \limsup_{y \lra \theta} \left | \frac{Ref(y)-Ref(\theta)}{y-\theta} \right | = +\infty$ and $\displaystyle \limsup_{y \lra \theta} \left | \frac{Imf(y)-Imf(\theta)}{y-\theta} \right | = +\infty$

\end{definition} 
\

\begin{lemma} \label{2.10.}
Let $\OO$ be a domain and $J \subset \bd{\OO}$ a relatively open subset of its boundary. If $K \subset \OO \cup J$ is a compact set, then there exists a larger compact set $K \subset E \subset \OO \cup J$, such that $E=\ch{E \cap \OO}$.
\end{lemma}

\begin{proof}
Obviously $K \cap J = K \cap \ch{\OO}$ is compact and disjoint from the closed set $\bd{\OO} \dif J$. Thus, $\dist(K \cap J, \bd{\OO} \dif J)>0$. We set $\e= \displaystyle \frac{1}{2} \dist(K \cap J, \bd{\OO} \dif J)$ and consider the set $E=K \cup \enosi_{\tau \in K \cap J} \ch{D(\tau,\e)} \cap \ch{\OO}$. We claim that $E$ is compact. Obviously, it suffices to prove that set $\enosi_{\tau \in K \cap J} \ch{D(\tau,\e)}$ is compact. Consider a sequence $\xn$ in $\enosi_{\tau \in K \cap J} \ch{D(\tau,\e)}$. Then, $x_n= \tau_n + \delta_n$, where $\tn$ is a sequence in $K \cap J$ and $|\delta_n| \leq \e$, for every $ n \in \N $. Therefore, we can find a convergent subsequences of $\tn$ and $\dn$, which implies that $\xn$ has a convergent subsequence in $\enosi_{\tau \in K \cap J} \ch{D(\tau,\e)}$. It follows easily from the way that $E$ was defined, that every point in $E$ can be approximated by points in $E \cap \OO$. Hence, the proof is complete. \\ \\
\end{proof}

\section{X(\boldmath ${\Omega}$,V) spaces in $\mathbb{C}$}
\justify
Let $\OO \subset \Co$ be a domain and $V \subset \OO$ an open set. We consider the set $X(\OO,V)=H(\OO) \cap H^{\infty}(V) = \{f \in H(\OO): f_{|V} \text{  is bounded}\}$. If $\overline{V} \subset \OO$ and $V$ is bounded, then obviously $X(\OO,V)=H(\OO)$ and the space is endowed with its usual Fr\'{e}chet topology. Furthermore, $\OO$ is always an $H(\OO)$-domain of holomorphy and the set of non-extendable functions in $H(\OO)$ is $G_{\delta}$ and dense in this space \cite{DoH},\cite{Range}. If $\overline{V} \subset \OO$ and $V$ is not bounded, we may have $X(\OO,V)\neq H(\OO)$ but again we can prove that $\OO$ is a $X(\OO,V)$-domain of holomorphy. Actually, this case is covered in the proof of Theorem \ref{3.1.} stated below.
\par Suppose $\overline{V} \cap \bd{\OO} \neq \emptyset$. The natural topology in this case is the Fr\'{e}chet topology induced by the seminorms $  \norm{f|_{K_m}}_{\infty}$ and $\norm{f|_{V}}_{\infty}$, where, $\{K_m\}_{m=1}^{\infty}$ is an exhaustive sequence of compact subsets of $\OO$. Obviously, $X(\OO,V)$ satisfies the requirements of Theorem \ref{2.6.}. Therefore, in order to prove that $\OO$ is a $X(\OO,V)$-domain of holomorphy it suffice to find $g_{b_1,b_2}=g \in X(\OO,V)$, for every pair of balls $(b_1,b_2)$, such that $
b_1\subset\overline b_1\subset b_2\cap\OO, \ \ b_2\cap\OO^c\neq\emptyset $, so that $g|_{b_1}$ does not admit any bounded holomorphic extension on $b_2$.\\ \\

\begin{theorem} \label{3.1.}
Let $\OO \subset \Co$ be a domain and $V \subset \OO$, an open set, such that $\overline{V} \cap \bd{\OO} \neq \emptyset$. We assume that for every $\zeta \in \overline{V} \cap \bd{\OO}$ there exists a sequence $\wn$ contained in $\overline{\OO}^{\mathrm{c}}$ with $w_n \lra \zeta$. Then, $\OO$ is a $X(\OO,V)$-domain of holomorphy and the set $\{f \in X(\OO,V): f\text{ is non-extendable}\}$ is dense and $G_{\delta}$ in $X(\OO,V)$.
\end{theorem}

\begin{proof}
Consider a pair of balls $(b_1,b_2)$ such that $\bb$. The set $b_2$ is connected, therefore $b_2 \cap \bd{\OO} \neq \emptyset$. Let $\zeta \in b_2 \cap \bd{\OO}$. If $\zeta \in \bd{\OO} \dif \overline{V}$, we can choose $g=g_{b_1,b_2} \in X(\OO,V)$ to be the function $\displaystyle g(z)=\frac{1}{z-\zeta} ,\: z \in \OO$. Then, $g$ is holomorphic on $\OO$ and bounded on $V$ since $\dist(\zeta,\overline{V})>0$. Thus, $g \in X(\OO,V)$ and $g|_{b_1}$ does not admit a bounded holomorphic extension on $b_2$, since $\zeta$ is a pole. Consider the case $\zeta \in \bd{\OO} \cap \overline{V}$. By our assumptions, there exist points of $\overline{\OO}^{\mathrm{c}}$ arbitrarily close to $\zeta$. Hence, we can find a point $w \in \overline{\OO}^{\mathrm{c}} \cap b_2$. We set $ \displaystyle g=g_{b_1,b_2}(z)=\frac{1}{z-w}, \: z \in \OO$. Similarly to the previous case, $g$ is holomorphic on $\OO$ and bounded on $V$, thus $g \in X(\OO,V)$, but $g|_{b_1}$ does not admit a holomorphic and bounded extension on $b_2$. \\
\end{proof}

\noindent Now, we consider some examples of pairs $(\OO,V)$, as in Theorem \ref{3.1.}, for whom the assumptions of the theorem are not satisfied and we examine whether the conclusion holds or not. 

\begin{example}  \label{3.2.}
Let $\OO = D(0,1) \dif \{0\} = V$. Clearly the point $0$ can not be approximated by points outside of $\overline{\OO}$. Consider euclidean balls so that $\bb,\: 0 \in b_2$. Then, if $g=g_{b_1,b_2} \in X(\OO,V)$, the point 0 is a removable singularity for $g$. Therefore, $g|_{b_1}$ has always a bounded holomorphic extension to $b_2$, As a result, $\OO$ is not a $X(\OO,V)$-domain of holomorphy in the weak sense, hence it is not a $X(\OO,V)$-domain of holomorphy.
\end{example}

\begin{example}  \label{3.3.}
A natural generalization of the previous example can be obtained by replacing $\{0\}$ with a compact set $A \subset D(0,1)$ with $\gamma(A)=0$, where $\gamma(A)$ denotes the Ahlfors capacity of the set $A$.
\end{example}

\begin{example} \label{3.4.}
Let $\OO=D(0,5) \dif [0,1]$ and $V=D(0,2) \dif [0,1]$. Again, the assumptions of Theorem 3.1. are not satisfied as $[0,1] \cap \tct = \emptyset$. We will show though, that $\OO$ is, indeed, a $X(\OO,V)$-domain of holomorphy. Let $(b_1,b_2)$ be a pair of euclidean balls, such that $\bb$ and suppose that $b_2 \cap [0,1] \neq \emptyset$. Choose a point $\beta \in (0,1) \cap b_2$. The M\"{o}bius  transformation $z \longmapsto  w(z)=\frac{z-\beta}{z}$ maps $[0,\beta]$ to the half-line $[-\infty,0]$. It is a well known fact, that the domain $\Co \dif [-\infty,0]$ admits a holomorphic branch of logarithm with: $-\pi < Im(logw) < \pi$ for all $w \in \Co \dif [-\infty,0]$. Let $ f(z)=log(\frac{z-\beta}{z})$ for $z \in \OO$ and $g(z)=e^{-if(z)}, \: z \in \OO$. In that case, $|g(z)|=|e^{-if(z)}|=e^{Re(-if(z))}=e^{Im(f(z))} \in (e^{-\pi},e^{\pi})$ for $z \in V$; thus, $g \in X(\OO,V)$. 
\par We will show that $g|_{b_1}$ does not admit admit a holomorphic extension on $b_2$. Suppose, by contradiction, that $F$ is a holomorphic extension of  $g|_{b_1}$ on $b_2$. Since, $b_2 \dif [0,\beta]$ is open and connected, the principle of analytic continuation implies that $F(z)=e^{-ilog\frac{z-\beta}{z}}$ for all $z\in b_2 \dif [0,\beta]$, hence $e^{-\pi} < |F(z)| < e^{\pi}$ for all $z\in b_2 \dif [0,\beta]$. Furthermore, the function F is assumed to be continuous on $b_2$, therefore by taking limits, we conclude that $F(z) \neq 0$ for all $z \in b_2$. 
\par Consider a smaller ball $b_3$ such that $\overline{b_3} \subset b_2, \: 0 \notin b_3, \: \beta \in b_3$. Then, $F(z) \neq 0$ for all $z\in b_3$, $b_3$ is a disc and $F|_{b_3}$ is holomorphic, therefore there exists a holomorphic branch of $logF$ on $b_3$, namely there exists a holomorphic branch of $log(\frac{z-\beta}{z})$ on $b_3$. Furthermore, $ 0 \notin b_3$, so there exists a holomorphic branch of $logz$ on $b_3$. This implies the existence of a holomorphic branch of $log(z-\beta)$ on $b_3$, which is absurd.
\par Finally, if $b_2 \cap [0,1]$, then $b_2$ intersects the boundary of $\OO$ on the circle C(0,5). By choosing, $\zeta \in b_2 \cap \bd{\OO}$ and $f_{\zeta}(z)=\frac{1}{z-\zeta}, \: z\in \OO$, we are done. Thus, the proof is complete. 
\begin{flushright} $\square$ \end{flushright}
\end{example}

\noindent We now proceed to studying a property of functions belonging to the class $X(\OO,V)$ for pairs $(\OO,V)$ satisfying some additional assumptions. 

\begin{theorem}  \label{3.5.}
1) Let $\OO=V$ be a Jordan domain, such that its boundary contains an open Jordan arc $J$, so that every compact subarc $J' \subset J$ is rectifiable.\\
2) Let $\OO$ be a domain and $V \subset \OO$, a Jordan domain such that $\overline{V} \cap \bd{\OO}$ contains an open Jordan arc $J$, such that every compact subarc $J'$ is rectifiable. We also assume that for every $\zeta \in J$ there exists a radius $r=r_{\zeta}>0$, such that $D(\zeta,r) \cap V = D(\zeta,r) \cap \OO$.\\
In  both cases 1 and 2, every $f \in X(\OO,V)$ has non-tangential limits almost everywhere in $J$, with respect to the arclength measure. 
\end{theorem}

\begin{proof}
Suppose (1) holds. Let $f \in X(\OO,V)$ and consider a compact subarc $J' \subset J$. By Lemma \ref{2.7.}, $J'$ can be extended to a rectifiable Jordan curve $\gamma$, such that the interior of $\gamma$ is contained in $V$. Let $G \subset V$ be the interior of $\gamma$ and fix a Riemann map $\phi:D \lra G$. Then, we have that $(f\circ \phi)\cdot \phi' \in H^{1}(D)$, because $f$ is bounded on $G$  and $\phi' \in H^{1}(D)$ by Theorem 3.12. of \cite{Duren}. By Theorem 10.1. of \cite{Duren} we have that $f \in E^{1}(G)$ and Theorem 10.3.  of \cite{Duren} gives us that $f$ has non-tangential limits almost everywhere on $J'$. Since, $J$ can be written as a countable union of compacts subarcs, the conclusion follows.
\par The proof of (2) is similar to the first one. Specifically, the same arguments yield the existence of $\displaystyle \mathrm{n.t.} \lim _{z \lra \zeta, z \in V} f(z)$ almost everywhere in $J$. The additional assumption that for every $\zeta \in J$ there exists a $r=r_{\zeta}>0$, such that $D(\zeta,r) \cap V = D(\zeta,r) \cap \OO$ yields that the $\displaystyle \mathrm{n.t.} \lim _{z \lra \zeta, z \in \OO} f(z)$ is essentially the same as the aforementioned, hence exists almost everywhere in $J$ with respect to the arclength measure on $J$. \\
\end{proof}

\section{X(\boldmath ${\Omega}$,V) spaces in $\mathbb{C}^{d}$}
\justify
In this section we consider the spaces $X(\OO,V)$ for $V \subset \OO \subset \Co^d$ where $\OO$ is a domain and $V$ is a bounded open subset of $\OO$. We give sufficient conditions so that $\OO$ is a $X(\OO,V)$-domain of holomorphy.

\begin{definition} \label{4.1.}
We say that an open connected subset of $\Co^d, \: (d \geq 1)$ satisfies the star condition if for every point $\zeta \in \bd{\OO}$ there exists a point $w \in \overline{\OO}^\mathrm{c}$, arbitrarily close to $\zeta$, and $v \in \Co^d$ a non-zero vector, such that the $(d-1)$-dimensional complex hyperplane $H=w+\{v\}^{\bot}=\{z \in \Co^d : <z,v>=<w,v>\} $ does not intersect $\overline{\OO}$ \cite{Georg}.
\par One can see that if $\OO \subset \Co^d$ satisfies the star condition, then $\OO$ is a Caratheodory domain, namely $\OO=\overline{\OO}^{\circ}$. If $d=1$, every Caratheodory domain $\OO \subset \Co$ satisfies the star condition, whereas in dimensions $d \geq 2$ one can prove that at least, convex open sets satisfy the condition \cite{Georg}.
\end{definition}

\begin{theorem}  \label{4.2.}
Let $\OO \subset \Co^d$ be a domain which satisfies the star condition and $V\subset \OO$ a bounded open set. Then, $\OO$ is a $X(\OO,V)$-domain of holomorphy.  
\end{theorem}

\begin{proof}
As we have previously discussed, it suffices to prove that for every pair of Euclidean balls $(b_1,b_2)$, such that $\bb$, there is a function $f \in X(\OO,V)$, such that $f|_{b_1}$ does not admit a bounded holomorphic extension on $b_2$.
\par The set $b_2$ is connected and intersect both $\OO$ and $\OO^\mathrm{c}$. Therefore, there exists a point $\zeta \in \bd{\OO} \cap b_2$. By the star condition, there exists a point $w \in b_2 \cap \overline{\OO}^\mathrm{c}$, $w=(w_1,w_2,...,w_d)$ and a complex hyperplane  $H=w+\{v\}^{\bot}$ of complex dimension $d-1$, where $v=(v_1,v_2,...,v_d)$ is a non-zero vector of $\Co^d$, such that $H$ does not intersect $\overline{\OO}$. Consider the function $z \longmapsto \displaystyle f(z)=\frac{1}{(z_1-w_1)v_1+...+(z_d-w_d)v_d}$ for $z \in H^\mathrm{c}$. Since $V$ is bounded, it follows easily that $f \in X(\OO,V)$. Suppose that $f_{|b_1}$ has a bounded holomorphic extension $F$ on $b_2$. The set $b_2 \cap H^\mathrm{c}$  is open and connected. This can be shown by counting real dimensions. Specifically, $b_2$ has real dimension $2d$, whereas $H$ has real dimension $2d-2$. Hence, the principle of analytic continuation implies that $F(z)=f(z)$ for all $z \in b_2 \cap H^{\mathrm{c}}$, which contradicts the fact that $F$ is bounded on $b_2$.

\end{proof}

\noindent Next, we present a second condition under which the conclusion of Theorem 4.2. remains valid. 

\begin{theorem}  \label{4.3.}
Let $\OO \subset \Co^d$ be an open pseudoconvex set and $V \subset \OO$ a bounded open set, such that $\overline{\OO}$ has a neighborhood basis of pseudoconvex open sets. Furthermore, we suppose that $\OO=\overline{\OO}^{\circ}$. Then, we conclude that $\OO$ is a $X(\OO,V)$-domain of holomorphy.
\end{theorem}

\begin{proof}
It suffices to prove that for every pair of Euclidean balls $(b_1,b_2)$, such that $\bb$, there exists a function $f \in X(\OO,V)$, so that $f_{|b_1}$ does not admit a bounded holomorphic extension on $b_2$. The condition $\OO=\overline{\OO}^{\circ}$ implies that $b_2 \cap \overline{\OO}^\mathrm{c} \neq \emptyset$. 
Let $\zeta \in b_2 \cap \overline{\OO}^\mathrm{c}$. By our assumptions there exist a pseudoconvex open set $G \supset \overline{\OO}$, such that $\zeta \notin G$. Let $Z$ be the connected component of $b_1$ in $b_2 \cap G$. Choose a point $A \in b_1$. Then, $A\ \in G,\: \zeta \notin G$, so $[A,\zeta] \cap \bd{G} \neq \emptyset$. Let $\sigma$ be the nearest point of the compact set $[A,\zeta] \cap \bd{G}$ to $\zeta$. Then, the segment $[A,\sigma)$ is contained in $G$ and specifically $[A,\zeta) \subset Z$. Since $G$ is pseudoconvex, if $\zn$ is a sequence in $[A,\sigma)$ converging to $\sigma$, there exists a holomorphic function $f:G \lra \Co$, such that $\displaystyle \sup_{n \in \N} |f(z_n)|=\infty$ \cite{Range}. The fact that $\overline{V} \subset \overline{\OO} \subset G$ implies that $f \in X(\OO,V)$. Suppose that $f_{|b_1}$ has holomorphic and bounded extension, $F$, on $b_2$. We have that $F(z)=f(z)$ for all $z \in b_1 \subset Z$ and $Z$ is open and connected. Therefore, $F(z)=f(z)$ for all $z \in Z$ by the principle of analytic continuation. Hence, $\displaystyle \sup_{n \in \N} |F(z_n)|= \sup_{n \in \N} |f(z_n)|=\infty$, which contradicts the fact that $F$ is bounded on $b_2$. 
\end{proof}

\section{Generalized Bergman and other spaces}
\justify
\par In this section we consider natural generalizations of spaces $X(\OO)$, we studied in sections 1 and 4. Under the assumptions of Theorem \ref{2.6.} and some additional ones we prove that the domain $\OO$ is a $X(\OO)$-domain of holomorphy for these new spaces $X(\OO)$. 
\par Let $\OO$ be a domain, $V \subset \OO$ an open subset of $\OO$ and $F \subset \{0,1,2,...\}$. The set $X(\OO,V,F)$ is the set of functions $f \in H(\OO)$, such that $f_{|V}^{(l)}$ is bounded for every $l \in F$. We equip $X(\OO,V,F)$ with the topology induced by the following seminorms: $\displaystyle \norm{f|_{K_m}}_{\infty},\: m=1,2,..., \norm{f|_{V}^{(l)}}_{\infty}, \: l \in F$, where $\{K_m\}_{m=1}^{\infty}$ is an exhaustive sequence of compact subsets of $\OO$. Clearly, the assumptions of Theorem \ref{2.6.} are satisfied.

\begin{corollary}  \label{5.1.}
If for every point $\zeta \in \overline{V} \cap \bd{\OO}$, there exist points of $\overline{\OO}^\mathrm{c}$ arbitrarily close to $\zeta$, then $\OO$ is a $X(\OO,V,F)$-domain of holomorphy and the set $\{f \in X(\OO,V,F): f \text{ is non-extendable}\}$ is $G_\delta$ and dense in $X(\OO,V,F)$.
\end{corollary}

\begin{proof}
Let $(b_1,b_2)$ be a pair of Euclidean balls, such that $\bb$. The set $b_2 \cap \bd{\OO}$ is non-empty. Choose a point $\zeta \in b_2 \cap \bd{\OO}$. If $\zeta \in \bd{\OO} \dif \overline{V}$, then for $ \displaystyle f_{\zeta}(z)=\frac{1}{z-\zeta}, \: z \in \OO, \: f^{(l)}$ remains bounded on $V$ for every $l \in F$, but $f$ does not admit a bounded holomorphic extension on $b_2$. If, on the other hand, $ \zeta \in \bd{\OO} \cap \overline{V}$, we choose a point $w \in b_2 \cap \overline{\OO}^\mathrm{c}$ and consider the function $\displaystyle f_w(z)=\frac{1}{z-w}, \: z \in \OO$. Similarly, to the previous argument $f^{(l)}$ is bounded on $V$ for every $l \in F$, but $f$ can not have a bounded and holomorphic extension on $b_2$, because the point $w$ is a pole.
\end{proof}

\noindent Another generalization is obtained if we replace the pair $(V,F)$ by a finite or infinite denumerable family of open subsets of $\OO$, $\displaystyle \{V_j\}_{j \in J}$ and assign to each $V_j$ a set $F_j \subset \{0,1,2,...\}$ demanding $f^{(l)}|_{V_j}$ be bounded for every $l \in F_j$. The space we obtain in this case is $\tomi_{j \in J} X(\OO,V_j,F_j)$ and its topology is induced by the seminorms $\norm{ f_{K_m}}_\infty, \: m=1,2,..., \: \norm{f^{(l)}}_{V_j}, \: l \in F_j, \: j \in J$. This space satisfies the requirements of Theorem 2.6., hence if we additionally assume that for all $j \in J$, $\overline{V_j} \cap \bd{\OO} \subset \tct$ we obtain that $\OO$ is a $\tomi_{j \in J} X(\OO,V_j,F_j)$-domain of holomorphy and the set of non-extendable functions is $G_\delta$ and dense in this space. The proof is similar to the one of Corollary \ref{5.1.} and is omitted.

\begin{remark}  \label{5.2.}
If $F_j=\{0\}$ for all $j \in J$, then the spaces $\bigcap_{j \in J} X(\OO,V_j)$ and $X(\OO, \bigcup_{j \in J} V_j)$ coincide if $J$ is finite, but might not be the same if $J$ is infinite. Generally, we have the inclusion $X(\OO, \bigcup_{j \in J} V_j) \subset \bigcap_{j \in J} X(\OO,V_j)$. We have already mentioned that both of those spaces satisfy the requirements of Theorem \ref{2.6.}. The sufficient conditions we provide for the conclusion of Theorem \ref{2.6.} to hold are equivalent to each other. Specifically, $\overline{ \enosi_{j \in J} V_j} \cap \bd{\OO} \subset \tct$ is equivalent to $\enosi_{j \in J} \overline{V_j} \cap \bd{\OO} \subset \tct$, because the intersection of $\tct$ with $\overline{\enosi_{j \in J} V_j}$ should be a closed set containing $\enosi_ {j \in J} \overline{V_j} \supset \enosi_{j \in J} V_j$. \\
\end{remark}

\noindent Let $\OO \subset \Co$ be a domain, $V \subset \OO$ a bounded open set and $p \in [1, \infty)$. Let $Y(\OO,V,p)=\{f \in H(\OO): \int_{V} |f|^{p} < \infty\}$.The topology of this space is the Fr\'{e}chet topology induced by the seminorms $\norm{f|_{K_m}}_{\infty}, \: m=1,2,...  \text{ and } \displaystyle \frac{1}{|V|}( \int_{V} |f|^{p} )^{\frac{1}{p}}, \: f \in Y(\OO,V,p)$, where $\{K_m\}_{m=1}^{\infty}$ is an exhaustive sequence of compact subsets of $\OO$. One can easily see that in this case $Y(\OO,V,p) \supset X(\OO,V)$; thus, if for every $\zeta \in \overline{V} \cap \bd{\OO}$ there exist points of $\overline{\OO}^\mathrm{c}$ arbitrarily close to $\zeta$, then $\OO$ is a $Y(\OO,V,p)$-domain of holomorphy. One can even consider the space $\displaystyle Z(\OO,V,p)=\{f \in H(\OO): \int_{V} |f|^{a} < \infty \text{  for all  } 1\leq a <p\}= \tomi_{1 \leq a < p} Y(\OO,V,p), \: (1<p)$, for whom the same results hold. The Fr\'{e}chet topology of $Z(\OO,V,p)$ is defined by the seminorms $\norm{f|_{K_m}}_{\infty}$, and $\displaystyle (\int_{V} |f|^{a_n})^{\frac{1}{a_n}}$, where $a_n$ is any strictly increasing sequence converging to $p$.\\

\noindent The last space we will discuss about will concern us further in the next section. Let $\OO \subset \Co$ be a domain and $ J \subsetneq \bd{\OO}$ a relatively open subset of its boundary. The set $A(\OO,J)$ contains all functions $f \in H(\OO)$, such that $f$ can be extended continuously on $\OO \cup J$. For $m=1,2,...$ we define the sets $\displaystyle \Delta_m=\{z \in \OO \cup J: \dist(z,\bd{\OO} \dif J) \geq \frac{1}{m}, \: |z| \leq m\}$. Then, the sets $\Delta_m$ are compact subsets of $\OO \cup J$ and every compact subset of $\OO \cup J$ is eventually contained in all of them. We equip $A(\OO,J)$ with the Fr\'{e}chet topology induced by the seminorms $\norm{f|_{\Delta_m}}_{\infty}, \: m=1,2,..., \: f \in A(\OO,J)$. The space $A(\OO,J)$ satisfies the requirements of Theorem \ref{2.6.} and if we additionally assume that every point in $J$ can be approximated by points in $\overline{\OO}^\mathrm{c}$, using similar arguments as before, we can prove that the set $\{f \in A(\OO,J): \: f \text{ is non-extendable}\}$ is $G_\delta$ and dense in $A(\OO,J)$. 
\par Furthermore, if $F \subset \{0,1,2,...\}$ we define $A(\OO,J,F)=\{f \in H(\OO): \: f^{(l)} \in A(\OO,J) \text{  for all  } l \in F\}$ and equip this set with the topology induced by the seminorms $\norm{f^{(l)}|_{\Delta_m}}, \: m=1,2,... \:, l \in F\cup\{0\}.$ Similar results hold for this space too. In particular, the results hold for the spaces $A^{p}(\OO,J), \: p\in \{0,1,2...\}\cup \{+\infty \}$. In the case $p < +\infty$ the space $A^{p}(\OO,J)$ corresponds to the set $F=\{0,1,2,...,p\}$. In the case $p=+\infty$ the set $F$ coincides with the set $\{0,1,2,...\}$. The reader can find the precise definition of the spaces $A^{p}(\OO,J)$ at the introduction of section 7, where we study those spaces elaborately.
\par Finally, we can combine any of the aforementioned spaces, considering functions which belong to some of them simultaneously. The topology in that case is the smallest topology which contains the topology of every space being considered. The resulting space satisfies the requirements of Theorem \ref{2.6.} and with the appropriate additional assumptions we can prove that the set of non-extendable functions is $G_\delta$ and dense.

\section{Nowhere differentiability in spaces $A(\boldmath {\Omega},J)$}

\subsection{The open unit disc}

Let $D=D(0,1)$ be the open unit disc and $J \subsetneq \bd{D}$ a relatively open subset of its boundary. Furthermore, we consider the sets $\displaystyle \Delta_m=\{z \in D \cup J: \: \dist(z,\bd{D} \dif J)\geq \frac{1}{m}\}$. As we have already mentioned in the previous section, every compact subset of $D \cup J$ is eventually contained in every $\Delta_m$. We equip the set $A(D,J)$ with the Fr\'{e}chet topology induced by the seminorms $\norm{f|_{\Delta_m}}_{\infty}, \: m=1,2,...$ and the set $A(D,J)$ becomes a complete metric space.\\
We note that a function defined on $J$ can be equivalently thought as a $2\pi$-periodic function defined on a suitable open set $J'$ of $\R$. Thus, by abuse of notation we will write $u(y)$ instead of $u(e^{iy}), \: y \in J'$ and refer to $J'$ simply as $J$.\\
Let $J_m=J \cap \Delta_m, \: m=1,2,...$. Then the sequence $\{J_m\}_{m=1}^{\infty}$ is an exhaustive sequence of compact subsets of $J$.\\
The result of this section is the following:

\begin{theorem}  \label{6.1.}
The set of functions $f \in A(D,J)$, such that $Ref|_{J},Imf|_{J}$ are not differentiable with respect to the parameter $\theta, \: \theta \in \R$, at any point $z=e^{i\theta} \in J$ contains a $G_\delta$ and dense set. \\
\end{theorem}

\noindent First of all, we state some definitions and lemmata which are needed for the proof of Theorem 6.1.

\noindent For $m,n \in \N$ we define the following sets:

\begin{equation}  D_n=\{ u \in C(J): \text{  for every  } \theta \in J \text{  there exists  } y \in (\theta - \frac{1}{n},\theta + \frac{1}{n})\cap J \text{  such that  } |u(y)-u(\theta)|>n|y-\theta| \} 
\end{equation}

\begin{equation}
E_n=\{f \in A(D,J): \: Ref|_{J} \in D_n\}
\end{equation}

\begin{equation}
D_{n,m}=\{ u \in C(J): \text{  for every  } \theta \in J_m \text{  there exists  } y \in (\theta - \frac{1}{n},\theta + \frac{1}{n})\cap J_m \text{  such that  } |u(y)-u(\theta)|>n|y-\theta| \}
\end{equation}

\begin{equation}
E_{n,m}=\{f \in A(D,J): \: Ref|_{J} \in D_{n,m}\}
\end{equation}

\noindent It easy to check that the aforementioned sets are related in the following way:\\
$\tomi_{m=1}^{\infty} D_{n,m} = D_n$ and $\tomi_{m=1}^{\infty} E_{n,m} = E_n$

\begin{lemma}  \label{6.2.}
For every $m,n \in \N$ the set $E_{n,m}$ is open in $A(D,J)$.
\end{lemma}

\begin{proof}
The proof is similar to the proof of Lemma 2.2. in \cite{Eske} and is omitted. 
\end{proof}

\begin{corollary}  \label{6.3.}
The set $S=\tomi_{m,n=1}^{\infty} E_{n,m}=\tomi_{n=1}^{\infty} E_{n}$ is a $G_\delta$ subset of $A(D,J)$.
\end{corollary}

\begin{proof}
The proof follows directly from Lemma 6.2.
\end{proof}

\begin{lemma}  \label{6.4.}
The set $S$ is a dense subset of $A(D,J)$.
\end{lemma}

\begin{proof}
We know that $S \neq \emptyset$ because it contains a complexification of the Weierstrass function \cite{EskeMakr}.
If $f_0 \in S$ and $p$ is a polynomial, then one can easily see that $f_0+p \in S$. Furthermore, it is true that the polynomials are dense in $A(D,J)$. Indeed, if $g \in A(D,J)$, then by Mergelyan's theorem ([Ru]), taking into account Lemma \ref{2.8.}, for every $m \in \N$ we can find a polynomial $p_m$ such that $\norm{(g-p_m)|_{\Delta_m}}_{\infty}<\frac{1}{m}$. The sequence $\{p_m\}_{m=1}^{\infty}$ converges to $g$ in the topology of $A(D,J)$. These two observations imply that the set of translations $\{f_0 +p: \: p \text{ polynomial}\} \subset S$ and is dense in $A(D,J)$.

\end{proof}

\noindent We proceed now to the proof of Theorem \ref{6.1.}

\begin{proof}
Lemmata \ref{6.2.} and \ref{6.4.} imply that the set $S$ is $G_\delta$ and dense subset of $A(D,J)$. Since multiplication by $i$ is an automorphism of $A(D,J)$, we conclude that the set $iS$ is also a $G_\delta$ and dense subset of $A(D,J)$. Hence, the set $R=S \cap iS$ is $G_\delta$ and dense by application of Baire's category theorem and consists of functions which are not differentiable with respect to the parameter $\theta, \: \theta \in \R$, at any point $z=e^{i\theta} \in J$. 
\end{proof}

\subsection{Jordan domains}

\noindent We consider now the case where $\OO$ is a Jordan domain and $J \subsetneq \bd{\OO}$ a relatively open subset of its boundary. The space $A(\OO,J)$ consists of all functions $f \in H(\OO)$, such that $f$ can be continuously extended on $\OO \cup J$. The topology in this space is induced by the seminorms $\norm{f|_{\Delta_m}}_{\infty}, \: m=1,2,...$ where $\Delta_m=\{z \in \OO \cup J: \dist(z, \bd{\OO} \dif J)\geq \frac{1}{m}\}$. We also define the set $A_{0}(\ext \dif \overline{\OO},J)$ to be the set of all functions $f \in H(\Co \dif \overline{\OO})$ such that $f$ vanishes to infinity and can be extended continuously on $(\Co \dif \overline{\OO}) \cup J$.The Fr\'{e}chet topology of this space is induced by the seminorms $\norm{f|_{\widetilde{{\Delta_m}}}}_{\infty},\: m=1,2,...$ where $\widetilde{{\Delta_m}}= \Delta_m \cap \overline{D(0,m)}$. In this section nowhere differentiability is meant with respect to the parametrization of $J$, which is induced by any Riemann map $\phi$. The extension of the Riemann map is guaranteed by the Osgood - Caratheodory theorem. Before we state the results of this section we recall the Definition  2.9. which states that if $L \subset \bd{D}$ is a relatively open subset of the unit circle, we say that a continuous function $f \in C(L)$ belongs to $Z_L$, if for every $\theta \in L$ we have that $\displaystyle \limsup_{y \lra \theta} \left | \frac{Ref(y)-Ref(\theta)}{y-\theta} \right | = +\infty$ and $\displaystyle \limsup_{y \lra \theta} \left | \frac{Imf(y)-Imf(\theta)}{y-\theta} \right | = +\infty$. 

\begin{proposition}  \label{6.5.}
Let $\OO$ be a Jordan domain and $J \subsetneq \bd{\OO}$ a relatively open subset of its boundary. Consider also  $\phi: \overline{D} \lra \overline{\OO}$, a Riemann map and the open set $L=\phi^{-1}(J) \subset \bd{D}$. The set of functions $f \in A(\OO,J) $ such that $(f \circ \phi)|_{L} \in Z_{L}$ is a $G_\delta$ dense subset of $A(\OO,J)$.
\end{proposition}

\begin{proposition}  \label{6.6.}
Let $\OO$ be a Jordan domain and $J \subsetneq \bd{\OO}$ a relatively open subset of its boundary. If $\phi: D \lra \ext \dif \OO$ is a Riemann map and $L=\phi^{-1}(J)$, the set of functions $f \in A_{0}(\ext \dif \overline{\OO},J)$, such that $(f \circ \phi)|_{L} \in Z_{L}$ is a dense and $G_\delta$ subset of $A_{0}(\ext \dif \overline{\OO},J)$.
\end{proposition}

\noindent The proofs of the Propositions 6.5. and 6.6. are similar to the proofs of Theorem 3.1., 3.2. in \cite{MastrPan} and thus, omitted.

\subsection{Domains bounded by a finite number of disjoint Jordan curves}
\noindent Let $\OO$ be a bounded domain whose boundary consists of a finite number of disjoint Jordan curves. If $V_0,V_1,...,V_{n-1}$ are the connected components of $\ext \dif \OO$, $\infty \in V_0$ and $\OO_0=\ext \dif V_0, \: \OO_1=\ext \dif V_1,...,\: \OO_{n-1}=\ext \dif V_{n-1}$, there exist Riemann maps $\phi_i:\overline{D} \lra \overline{\OO}_i,\:  i \in I=\{0,1,...,n-1\}$. Additionally, we consider $J_i \subset \bd{\OO_i}, \: i\in I$, relatively open and $\phi_i^{-1}(J_i)=L_i \subset \bd{D}$, relatively open subsets of the unit circle, such that there exist at least an $i\in I$ for whom $J_i$ is distinct from $\bd{\OO_i}$. Let $J=J_0 \cup J_1\cup...\cup J_{n-1}$. We define $A(\OO,J)=\{f \in H(\OO): f \text{ can be extended continuously on } J_i, i \in I\}$. Consider also the sets $\displaystyle {\Delta}_{m}^{(i)}=\{z \in \OO_i \cup J_i\ : \dist(z,\bd{\OO_i} \dif J_i) \geq \frac{1}{m}\}$ for $i \in I$ and $m \in \N$. Here, we use the convention, that if $\bd{\OO_i} \dif J_i$ is an empty set, then  ${\Delta}_{m}^{(i)}=\overline{\OO_i} \text{ for every } m \in \N$. Lemma \ref{2.8.} implies that for fixed $i \in I$ the sets ${\Delta}_{m}^{(i)}$ have connected complement in the Riemann sphere. Let $\Delta_m = \tomi_{i \in I} {\Delta}_{m}^{(i)}, \; m \in \N$. The sequence $\{\Delta_m\}_{m=1}^{\infty}$ consists of compact subsets of $\OO \cup J$, such that every compact set of $\OO \cup J$ is eventually contained to every $\Delta_m$. Moreover, $\ext \dif \Delta_m = \enosi_{i=0}^{n-1} \ext \dif {\Delta}_{m}^{(i)}$. For every $i \in I$, the set $\ext \dif {\Delta}_{m}^{(i)}$ is open, connected, therefore $\ext \dif \Delta_m$ has at most $n$ connected components. If $m$ is sufficiently large, the number of components is exactly $n$, each of whom contains a connected component of $\ext \dif (\OO \cup J)$. Finally, we set ${L}_{m}^{(i)}=\phi_{i}^{-1}({\Delta}_{m}^{(i)} \cap J_i)$. For fixed $i \in I$, the sets ${L}_{m}^{(i)}$ form a sequence of compact subsets of $L_i$, such that every compact subset of $L_i$ is eventually contained in all of them. We equip the set $A(\OO,J)$ with the Fr\'{e}chet topology induced by the seminorms $\norm{f|_{\Delta_m}}_{\infty}, \:m=1,2,..., \: f \in A(\OO,J)$.
\par We are interested in functions $f \in A(\OO,J)$ for whom $(f \circ \phi_i)|_{L_i} \in Z_{L_i}$ for every $i \in I$. For $k \in \N$ we set $D^{(i)}_k=\{ u \in C(L_i): \text{ for all }\theta \in L_i \text{ there exists } y \in (\theta-\frac{1}{k},\theta+\frac{1}{k})\cap L_i : \: |u(y)-u(\theta)|>k|y-\theta| \}$ and $E^{(i)}_k=\{ f \in A(\OO,J): Re(f \circ \phi_i)|_{L_i} \in D^{(i)}_k\}$. Consider also the sets $D^{(i)}_{k,m}=\{ u \in C(L_i): \text{ for all }\theta \in {L_m}^{(i)} \text{ there exists } y \in (\theta-\frac{1}{k},\theta+\frac{1}{k})\cap {L}_{m}^{(i)} : \: |u(y)-u(\theta)|>k|y-\theta| \}$ and $E^{(i)}_{k,m}=\{ f \in A(\OO,J): Re(f \circ \phi_i)|_{L_i} \in D^{(i)}_{k,m}\}$. Similarly to what we have done in the previous cases, we can show that $\tomi_{m=1}^{\infty} D^{(i)}_{k,m}=D^{(i)}_k$ and $\tomi_{m=1}^{\infty} E^{(i)}_{k,m}=E^{(i)}_k$. Furthermore, for every $k,m \in \N$ and $i \in I$, the set $E^{(i)}_{k,m}$ is open in $A(\OO,J)$ . The proof of the last statement is similar to that of Lemma \ref{6.2.} and thus omitted.

\begin{lemma}  \label{6.7.}
The set $S=\tomi_{i=0}^{n-1} \tomi_{k,m=1}^{\infty} E^{(i)}_{k,m}$ is dense in $A(\OO,J)$.
\end{lemma}

\begin{proof}
We only prove that the set $S^{(0)}=\tomi_{k,m=1}^{\infty} E^{(0)}_{k,m}$ is dense in $A(\OO,J)$. Let $f \in A(\OO,J)$. By the Laurent decomposition \cite{CoNePa} there exist $f_i \in A(\OO_i,J_i)$, such that $f=f_0+f_1+...+f_{n-1}$. We set $g=f_1+...+f_{n-1}$, hence we have that $f=f_0+g$. The argument we present afterwards is a modification of the proof of \cite{MastrPan}.
\par Let $m \in \N$ and consider the set $\Delta_m$. Then, $\Delta_m \subset \OO_0 \cup J_0 \subset \overline{\OO_0}$ and $\ext \dif \Delta_m$ has $n$ connected components, each of whom contains a connected component of $\ext \dif (\OO \cup J)$. Specifically, it is true that for $i \in I$, $V_i\cup (\bd{\OO_i} \dif J_i) \subset \ext \dif {\Delta}_{m}^{(i)}$. The function $ \phi_{0}^{-1}: \overline{\OO_0} \lra \overline{D}$ is a homeomorphism, thus the set ${\phi}_{0}^{-1}(\Delta_m) \subset D \cup L_0 \subset \overline{D}$ is compact and $\ext \dif {\phi}_{0}^{-1}(\Delta_m)$ has $n$ connected components each of whom contains a connected component of $\ext \dif {\phi}_{0}^{-1}(\OO \cup J) $. Hence, there exists a rational function $r_m$ with poles off ${\phi}_{0}^{-1}(\Delta_m)$, such that $\norm{g \circ \phi_0 -r_m}_{{\phi}_{0}^{-1}(\Delta_m)} < \frac{1}{m}$. By Runge's pole sliding theorem we can assume that the poles of $r_m$ are contained in $\ext \dif {\phi_0}^{-1}(\OO \cup J)$, therefore we may assume that $r_m \circ {\phi}_{0}^{-1} \in A(\OO,J)$. By repeating this procedure for $m=1,2,...$ we obtain a sequence of rational functions $\{r_m\}_{m=1}^{\infty}$, such that $r_m \circ {\phi}_{0}^{-1} \in A(\OO,J)$ and converges to $g$ in the topology of $A(\OO,J)$. Consider the function $f_0 \in A(\OO_0,J_0)$. Then, according to Proposition \ref{6.5.}, there exists a sequence $\{g_m\}_{m=1}^{\infty}$ in $A(\OO_0,J_0)$, such that $g_m \circ \phi_0 \in Z_{L_0}$ for $m=1,2,...$, which converges to $f_0$. Therefore, the sequence $\{g_m+r_m \circ {{\phi_0}^{-1}}\}_{m=1}^{\infty}$ converges to $f$, in the topology of $A(\OO,J)$. For $y \neq \theta$ in the same connected component of $L_0$ we have that, $\displaystyle \left | \frac{Re(g_m+r_m \circ {\phi}_{0}^{-1})(\phi_0(y))-Re(g_m+r_m \circ {\phi}_{0}^{-1})(\phi_0(\theta))}{y-\theta} \right | \geq \left | \frac{Re(g_m)(\phi_0(y))-Re(g_m)(\phi_0(\theta))}{y-\theta} \right | - \left | \frac{Re(r_m)(y)-Re(r_m)(\theta)}{y-\theta} \right |$ by the triangle inequality. Since the poles of $r_m$ are off $\phi_0^{-1}(\OO \cup J)$, we have that $r_m$ is differentiable on $L_0$. Therefore, taking the $\limsup$ as $y \lra \theta^{+}$ yields that the sequence $\{g_m+r_m \circ {{\phi_0}^{-1}}\}_{m=1}^{\infty}$ is contained in $S^{(0)}=\tomi_{k,m=1}^{\infty} E^{(0)}_{k,m}$. Hence, Baire category theorem implies that $S^{(0)}$ is a $G_\delta$ and dense set in $A(\OO,J)$, because it is a countable intersection of open dense sets. 
\par In a similar way we prove that the sets $S^{(i)}$ are $G_\delta$ and dense in the complete space $A(\OO,J)$. Baire's theorem applied once more implies that the set $S=\tomi_{i=0}^{n-1} S^{(i)}$ is also dense and $G_\delta$ in $A(\OO,J)$. The proof is complete.
\end{proof}

\begin{remark}  \label{6.8.}
The arguments used in the previous proofs imply easily the following:
\item[1.] If $\OO$ is a Jordan domain and $J \subset \bd{\OO}$ is a relatively open set, then the polynomials are dense in the space $A(\OO,J)$. 
\item[2.] If $\OO$ is a bounded domain whose boundary consists of a finite number of disjoint Jordan curves and $J$ is a relatively open subset of $\bd{\OO}$, then the rational functions with poles off $\OO \cup J$ are dense in $A(\OO,J)$. In fact, we can fix $n$ poles, each one in one hole of $\ext \dif \overline{\OO}$. \\
In 1, we use Mergelyan's theorem \cite{Ru}. In 2 we need an extension of it, where $K^{\mathrm{c}}$ has a finite number of components \cite{Ru}.
\end{remark}

\section{Nowhere differentiability in spaces  $A^{p}({\Omega},J)$}

Let $\OO$ be a simply connected domain, $\OO \neq \Co$. It is well known that there exists a Riemann map $\phi: D \lra \OO$. Suppose also, that there exist a relatively open subset $L$ of the unit circle and a relatively open subset $J$ of $\bd{\OO}$, such that there is a homeomorphism $\widetilde{\phi}: D \cup L \lra \OO \cup J$ which extends $\phi$. 
\par We define the space $A^{p}(\OO,J)$ similarly to the space $A^{0}(\OO,J)=A(\OO,J)$. Specifically, a function $f$ belongs to $A^{p}(\OO,J)$ if $f$ is holomorphic on $\OO$ and for every $0 \leq l \leq p, \: l\in \N \cup \{0\}$ the derivative of order $l$ belongs to $A(\OO,J)$. The topology in this space is induced by the seminorms $\norm{f^{(l)}|_{\Delta_m}}_{\infty}, \: m=1,2,...$ where $\{\Delta_m\}_{m=1}^{\infty}$ is the usual sequence of compact subsets of $\OO \cup J$. We are interested in functions $f \in A^{p}(\OO,J)$ such that $Re(f^{(p)} \circ \widetilde{\phi})|_{L}, Im(f^{(p)} \circ \widetilde{\phi})|_{L}$ are nowhere differentiable.

\underline{We assume the following:}

\item[1.] The space $A^{0}(\OO,J)$ contains a $G_\delta$ dense subset of functions $f$, such that $Re(f \circ \widetilde{\phi})|_{L},\:  Im(f \circ \widetilde{\phi})|_{L}$ are nowhere differentiable.

\item[2.] For every $\zeta \in J$ there exist $r>0$ and $C=C_{\zeta}>0$, such that for every $z,w \in D(\zeta,r) \cap \OO$, there exists a rectifiable curve $\gamma_{z,w} \subset D(\zeta,r) \cap \OO$ joining $z$ and $w$, such that $\mathbf{length}(\gamma_{z,w}) \leq C|z-w|$ and $\overline{D(\zeta,r) \cap \OO} \subset \OO \cup J$.

\item[3.] For every compact set $K \subset \OO \cup J$, there exist a positive constant $M=M_K$ and a compact set $L \subset \OO \cup J$, such that for every $z,w \in K \cap \OO$ there exists a rectifiable curve $\gamma_{z,w} \subset L \cap \OO$ joining $z$ and $w$ with $\mathbf{length}(\gamma_{z,w}) \leq M_K$. See also, \cite{NestZad}, \cite{SmStVo}.

\begin{theorem}  \label{7.1.}
If the assumptions 1,2 and 3 hold for the simply connected domain $\OO$ then for every $p \in \{0,1,2,...\}$ there is a set $S_p \subset A^{(p)}(\OO,J)$, $G_\delta$ and dense in $A^{(p)}(\OO,J)$, such that for every $f \in S_p$ the functions $Re(f^{(p)} \circ \widetilde{\phi})|_{L}, Im(f^{(p)} \circ \widetilde{\phi})|_{L}$ are nowhere differentiable.
\end{theorem}

\begin{proof}
Let $f \in A^{0}(\OO,J)$ and consider the function $\displaystyle F(z)=\int_{\gamma_{a,z}} f(\zeta) d\zeta, \: z \in \OO$, where $a$ is a fixed point of $\OO$ and $\gamma_{a,z}$ is a rectifiable curve in $\OO$ joining $a$ and $z$. The function $F$ is well defined, because of the independence of the path of integration in the simply connected domain $\OO$. We will prove that there is a unique continuous extension of $F$, $\widetilde{F}$ on $\OO \cup J$. For that purpose consider a point $\zeta \in J$. By the second assumption there exists a radius $r>0$ and a constant $C=C_\zeta>0$, such that for every $z,w \in D(\zeta,r) \cap \OO$ there exists a rectifiable curve $\gamma_{z,w} \subset L \cap \OO$ joining $z$ and $w$ with $\mathbf{length}(\gamma_{z,w}) \leq M_K$. Hence, we have that if $z,w \in D(\zeta,r) \cap \OO$, then $|F(z)-F(w)|=\left | \int_{\gamma_{a,z}} f(\zeta) d\zeta, \: z \in \OO \right | \leq \norm{f|_{D(\zeta,r) \cap \OO}}_{\infty} \cdot C \cdot |z-w|$ and $\norm{f|_{D(\zeta,r) \cap \OO}}_{\infty}<+\infty$ because $\overline{D(\zeta,r) \cap \OO} \subset \OO \cup J$ and is a compact set. Therefore, $F$ is Lipschitz continuous in a neighborhood of $\zeta \in \OO$, hence it can be uniquely extended continuously in $\zeta$. Since, $\zeta$ was arbitrarily chosen, $F$ can be continuously extended on $\OO \cup J$. Hence, if $f \in A^{0}(\OO,J)$, we have that $F \in A^{1}(\OO,J)$.
\par We will now prove that the map $A^{0}(\OO,J) \ni f \lra F \in A^{0}(\OO,J)$ is continuous. Obviously, it is a linear map. Let $K$ be a compact set contained in $\OO \cup J$. By lemma 2.10. there exists a larger compact set $K \subset E \subset \OO \cup J$ such that $\overline{E\cap\OO}=E$. Without loss of generality, we can assume that the base point $a$ is contained in $K$. By the third assumption, there is a constant $M=M_{E}>0$ and a compact set $L=L_{E}\subset \OO \cup J$, such that any two points $z,w \in E\cap \OO$ can be joined with a rectifiable curve $\gamma_{z,w} \subset L \cap \OO$ with $\mathbf{length}(\gamma_{z,w}) \leq M_E$. Therefore, we have $\displaystyle \sup_{z\in K} |F(z)| \leq \sup_{z \in E} |F(z)|=\sup_{z \in E\cap \OO} |F(z)|=\sup_{z \in E\cap \OO} \left|\int_{\gamma_{a,z}} f(\zeta) d\zeta \right | \leq M_E \cdot \sup_{\zeta \in L}|f(\zeta)|$ and obviously $\displaystyle \sup_{z \in K} |F'(z)|=\sup_{z\in K}|f(z)|$. Thus, the map $A^{0}(\OO,J) \ni f \lra F \in A^{0}(\OO,J)$ is continuous.
\par Finally, we can prove the theorem. Consider the function $T: A^{0}(\OO,J) \times \Co \lra A^{1}(\OO,J)$ which maps a pair $(f,w)$ to the function $F+w$. One can easily see that $T$ is linear, bijective, and it follows from our last argument that $T$ is also continuous. The spaces $A^{0}(\OO,J) \times \Co, A^{1}(\OO,J)$ are Fr\'{e}chet spaces, hence the open mapping theorem \cite{Rud} suggests that $T$ is a linear isomorphism. Therefore, the image of the set $S_0 \times \Co$ is a $G_\delta$ dense set in $A^{1}(\OO,J)$ and consists of functions $g \in A^{1}(\OO,J)$, such that $Re(g' \circ \widetilde{\phi})|_{L}, Im(g' \circ \widetilde{\phi})|_{L}$ are nowhere differentiable. By using the same argument successively, we can prove the same result for $p>1$.
\end{proof}

\begin{corollary}  \label{7.2.}
If $\OO$ is a bounded convex domain and $J\subsetneq \bd{\OO}$ a relatively open subset of its boundary, the set of functions $f \in A^{p}(\OO,J)$, such that  $Re(f^{(p)} \circ \widetilde{\phi}), Im(f^{(p)} \circ \widetilde{\phi})$ are nowhere differentiable on $L=\phi^{-1}(J)$, where $\phi : \overline{D} \lra \overline{\OO}$ is a Riemann map, contains a $G_\delta$ and dense set.
\end{corollary}

\begin{proof}
Proposition \ref{6.5.} yields that the first assumption of Theorem \ref{7.1.} is true for $A^{0}(\OO,J)$. Moreover, the second assumption is also true, because by the convexity of $\OO$ we can set $C_\zeta=1$ for every $\zeta \in J$. Finally, let $K \subset \OO \cup J$ be a compact set. We set $L=\enosi_{z \in K} [a,z]$, where $a \in \OO$ is fixed. This set is also compact, because $K$ and $[0,1]$ are compact. If $z,w$ are points in $K$, then they can be joined with the rectifiable curve $[z,a] \cup [a,w]$. Hence, we can set $M_K$ to be $\sup_{z,w\in K} (|a-z|+|a-w|)$. The last supremum is finite because $K$ is compact. Since, the assumptions 1,2 and 3 are true, Theorem \ref{7.1.} yields the result.

\end{proof}

\noindent Now, consider the case where $\OO \neq \Co$ is an unbounded convex domain and $J\subsetneq \bd{\OO}$ a relatively open subset of its boundary. There exists a Riemann map $\phi: \overline{D} \lra \overline{\OO}\cup\{\infty\}$. Let $a$ be a point, such that $a \notin \overline{\OO}$. The M{\"o}bius transformation $z \longmapsto \mu(z)=\frac{1}{z-a}$ is an automorphism of the extended plane $\ext$. It  maps $\OO$ to a Jordan domain $V$ and $J$ to a relatively open subset of $\bd{V}, \:J'$, such that $0=\mu(\infty) \in \bd{V} \dif J'$. Furthermore, the function $\mu \circ \phi: \overline{D} \lra \overline{V}$ is a Riemann map. Consider the function $T: A^{0}(V,J') \lra A^{0}(\OO,J)$, which maps $f \in A^{0}(V,J')$ to $f \circ \mu$. Then, $T$ is linear, bijective and continuous. The continuity follows from : $\norm{(f \circ \mu)|_{\Delta_m}}_{\infty}=\norm{f|_{\mu(\Delta_m)}}_{\infty}$ because $\mu(\Delta_m)$ is a compact set. Hence, by the open mapping theorem, the map $T$ is a linear isomorphism between the Fr\'{e}chet spaces $A^{0}(V,J')$ and $A^{0}(\OO,J)$, therefore the image of the set $S_0 \subset A(V,J')$, whose existence is guaranteed by Proposition 6.4., is a dense and $G_\delta$ set. Moreover, $T(S_0)$ consists of functions $g=f \circ \mu, \: f \in A(V,J')$, such that $Re(f \circ \mu \circ \phi)|_{L}, \: Im(f \circ \mu \circ \phi)|_{L}$ are nowhere differentiable, namely $Re(g \circ \phi)|_{L}, Im(g \circ \phi)|_{L}$ are nowhere differentiable. Therefore, the first assumption of Theorem 7.1. is true. Using similar arguments as in Corollary \ref{7.2.}, we can prove that assumptions 2 and 3 are also valid in this case. In conclusion, we have the following corollary:

\begin{corollary}  \label{7.3.}
If $\OO$ is a \underline{unbounded} convex domain and $J\subsetneq \bd{\OO}$ a relatively open subset of its boundary, the set of functions $f \in A^{p}(\OO,J)$, such that  $Re(f^{(p)} \circ \widetilde{\phi}), Im(f^{(p)} \circ \widetilde{\phi})$ are nowhere differentiable on $L=\phi^{-1}(J)$, where $\phi : \overline{D} \lra \overline{\OO}$ is a Riemann map, contains a $G_\delta$ and dense set.
\end{corollary}

\begin{remark}
Using conditions analogous to assumptions 1,2,3 we can prove that for every convex domain $\OO$, the set of polynomials is dense in $A^{p}(\OO,J), \: p \in \{0,1,2,...\}\cup\{\infty\}$. What are possible generalizations of this fact?
\end{remark}

\section{Nowhere differentiability in $A({\Omega},J)$ with respect to the position}

Consider the space $A(\OO,J)$, where $\OO$ is a Jordan domain and $J \subsetneq \bd{\OO}$, a relatively open subset of its boundary. Let $\displaystyle S(\OO,J)=\{f \in A(\OO,J): \limsup_{z \ra z_0,z \in J} \left |\frac{f(z)-f(z_0)}{z-z_0} \right | =+\infty \text{ for every } z_0 \in J\}$. If the class $S(\OO,J)$ is non-empty, it contains functions that are not differentiable, with respect to the position, at any point $z\in J$. Here, nowhere differentiability with respect to the position means that for every point $z_0 \in J$, the limit of the quotient $\displaystyle \left |\frac{f(z)-f(z_0)}{z-z_0} \right |$ as $z \ra z_0, (z \in J \dif \{z_0\})$ does not exist in $\Co$. Using the fact that the polynomials are dense in the space $A(\OO,J)$, we will prove that either $S(\OO,J)$ is void or it is a $G_\delta$ and dense set. We note that if the parametrization induced by any Riemann map $\phi: \overline{D} \lra \overline{\OO}$ is smooth, with non-vanishing derivative in $J$, then Proposition 6.5. yields that $S(\OO,J)$ is non-empty and in fact, $G_\delta$-dense in $A(\OO,J)$. \\
For $m,n \in \N$ we consider the sets $\displaystyle E_{n,m}=\{f \in A(\OO,J):  \text{ for all } z_0 \in J_m \text{ there exists a point } z \in (J_m\dif \{z_0\}) \cap D(z_0,\frac{1}{n})  \text{ such that } \left | \frac{f(z)-f(z_0)}{z-z_0} \right | > n \}$, where $J_m=J \cap \Delta_m, \: m=1,2,...$ .

\begin{lemma}  \label{8.1.}
For every $m,n \in \N$ the set $E_{n,m}$ is open in $A(\OO,J)$.
\end{lemma}

\begin{theorem}  \label{8.2.}
If the set $S(\OO,J)$ is non-empty, then it is $G_\delta$-dense in $A(\OO,J)$. Hence, the set of functions that are not differentiable at any point of $J$, with respect to the position, contains a $G_\delta$-dense set. 
\end{theorem}

\noindent The proofs of Lemma 8.1. and Theorem 8.2. are similar to the proofs we present in Lemmata \ref{6.2.}, \ref{6.3.} ,thus they are omitted. We note that Lemma 8.1. and Theorem 8.2. are analogous to results stated in \cite{KavvMakr}. 

\begin{remark}  \label{8.3.}
In a private communication, Christoforos Panagiotis proved that for every Jordan domain $\OO$, it holds that $S(\OO,\bd{\OO}) \neq \emptyset$; this implies obviously that $S(\OO,J) \neq \emptyset$. Combining that with the above, we conclude that $S(\OO,J)$ is $G_\delta$ and dense in $A(\OO,J)$. \\
\end{remark}

\noindent \textbf{Acknowledgments}: Some results of the present paper relate to
discussions held during a Research in pairs program at
Cirm-Lumini on May 2017.

\bigskip

\noindent \textsc{Dimitris \ Lygkonis}: Department of
Mathematics, National and Kapodistrian University of Athens, Panepistimiopolis 157-84,
Athens, Greece.

\smallskip

\noindent \textit{E-mail:} \texttt{dimitris.ligonis@gmail.com}
\bigskip

\noindent \textsc{Vassilis \ Nestoridis}: Department of
Mathematics, National and Kapodistrian University of Athens, Panepistimiopolis 157-84,
Athens, Greece.

\smallskip

\noindent \textit{E-mail:} \texttt{vnestor@math.uoa.gr}
\end{flushleft}


\begin{thebibliography}{}
\bibitem{CoNePa}G. Costakis, V. Nestoridis, I. Papadoperakis, Universal Laurent Series, Proc. Edinb. Math. Soc. (2) 48 (2005), no 3,571-583

\bibitem{Duren}P. L. Duren, Theory of $H^{p}$ spaces, Academic press, N.Y. and London, 1970

\bibitem{Eske}A. Eskenazis, Topological genericity of nowhere differentiable functions in the disc algebra, Arch. Math. (Basel) 103 (2014) no 1, 85-92, see also arxiv:1311.0142

\bibitem{EskeMakr}A. Eskenazis, K. Makridis, Topological genericity of nowhere differentiable functions in the disc and polydisc algebras, JMAA 420 (2014) no 1, 435-446, see also arxiv:1311.1176

\bibitem{Georg}N. Georgakopoulos, Holomorphic extendability in $\Co^n$ as a rare phenomenon, arxiv:1611.05367

\bibitem{GeorgMastrNest}N. Georgakopoulos, V. Mastrantonis, V. Nestoridis, Relations of the spaces $A^{p}(\OO)$ and $C^{p}(\bd{\OO})$, Results in Mathematics, \textit{to appear}, DOI:10.17863/CAM.21571 .

\bibitem{KavvMakr}K. Kavvadias, K. Makridis, Nowhere diffferentiable functions with respect to the position, submitted, see also arxiv:1701.04875

\bibitem{LioNest}V. Liontou, V. Nestoridis, Jordan domains with a rectifiable arc in their boundary, submitted, see also arxiv:1705.02254

\bibitem{MastrPan}V. Mastrantonis , C. Panagiotis, Nowhere differentiable functions of analytic type on products of finitely connected planar domains, 2017, Monatsh.Math., DOI:10.10007/s00605-017-1128-8, see also arxiv:1608.08235


\bibitem{DoH}V. Nestoridis, Domains of Holomorphy, Annales math\'{e}matiques du Qu\'{e}bec, 42 (2018) no. 1,
101-105, see also arxiv:1701.00734


\bibitem{NestZad}V. Nestoridis, I. Zadik, Pad\'{e} approximants, density of rational functions in $A^{\infty}(\OO)$ and smoothness of the integration operator, JMAA 423 (2015) no 2, 1514-1539, see also arxiv:1212.4394

\bibitem{Range}R. M. Range, Holomorphic functions and Integral Representations in Several Complex Variables, Graduate Texts in Math.108, Springer-Verlag, N.Y. 1998.



\bibitem{Ru}W. Rudin, Real and Complex Analysis, (3rd edition, Mc Graw-Hill Inc.1967)


\bibitem{Rud}W. Rudin, Functional Analysis, (2nd edition, Mc Graw-Hill Inc. 1991)

\bibitem{SmStVo}W. Smith, D. M. Stolyarov, A. Volberg,  Uniform
approximation of Bloch functions and the boundedness of the integration
operator on $H^\infty$.
Adv.Math.314 (2017)185-202,  see also arxiv:1604.05433.


\end{thebibliography}
\end{document}